\documentclass[a4paper,10pt]{article}

\usepackage{yhmath}
\usepackage{amssymb,amsthm}  
\usepackage{amsfonts,enumerate}
\usepackage{amsmath}
\newcommand{\g}{\mathfrak{g}}
\newcommand{\h}{\mathfrak{h}}

\newcommand{\Gr}{\mathbb{G}}
\newcommand{\co}{\mathbb{C}}

\newcommand{\Or}{\mathcal{O}}
\newcommand{\rk}{\mathrm{rank}\,}
\newcommand{\coker}{\mathrm{coker}\,}
\newcommand{\bor}{\mathfrak{b}}
\newcommand{\Sp}{\mathfrak{sp}(4)}
\renewcommand{\sl}{\mathfrak{sl}(3)}
\newcommand{\Clif}{\mathrm{Cliff}(\g\oplus \g^\vee, ev)}
\title{Equations of some wonderful compactifications}
\author{Pascal Hivert\thanks{Laboratoire de Math\'ematiques de Versailles, Universit\'e de Versailles--Saint-Quentin, email: pascal.hivert@math.uvsq.fr } }

\begin{document}

\maketitle

\begin{abstract}
De Concini and Procesi have defined in \cite{CP} the wonderful compactification $\bar{X}$ of a symmetric space $X=G/G^\sigma$ where $G$ is a semisimple adjoint group and $G^\sigma$ the subgroup of fixed points of $G$ by an involution $\sigma$. It is a closed subvariety of a grassmannian of the Lie algebra $\mathfrak{g}$ of $G$. In this paper, we prove that, when the rank of $X$ is equal to the rank of $G$, the variety is defined by linear equations. The set of equations expresses the fact that the invariant alternate trilinear form $w$ on $\g$ vanishes on the $-1$-eigenspace of $\sigma$.   
\end{abstract}
\bigskip

\section{Introduction}

Throughout this paper, the Lie algebras, the vector spaces and the projective spaces are defined over the complex field $\co$. Let $\g$ be a semisimple Lie algebra with adjoint group $G$, and $\kappa$ be a Killing form on $\g$. The trilinear alternate form $w\, : \, (x,y,z)\mapsto \kappa([x,y],z)$ is invariant under the adjoint action: it is an element of $(\bigwedge^3 \g^\vee)^G$. We put $g=\dim \g$, $l= \rk \g$ and $d=\frac{g+l}2$.

Let $\sigma$ be an involution of $G$, and $H=G^\sigma$ be the closed subgroup consisting of fixed points by $\sigma$.  The \textit{rank of the symmetric space $X=G/H$} is the maximal dimension of the $-1$-eigenspace of $\sigma$ acting on a $\sigma$-invariant Cartan subalgebra of $\g$ ($\sigma$ induces an involution on the Lie algebra $\g$, denoted again by $\sigma$, moreover this involution preserves Killing forms on $\g$).

\medskip

In Section 6 of \cite{CP}, the \textit{minimal wonderful compactification} $\bar{X}$ of $X$ is defined as the closure in the grassmannian $\Gr( \dim \g^\sigma , \g)$ of the $G$-orbit of the point $\g^\sigma$, the Lie algebra of $G^\sigma$.
The action of $G$ in $\bar{X}$ has the following properties.
\begin{enumerate}
\item[1.] The variety $\bar{X}$ is a union of finitely number of $G$-orbits.
\item[2.] The set $\bar{X}\smallsetminus G\cdot \g^\sigma$ is an union of $r$ hypersurfaces $S_i$, $i\in\{1,\ldots,r\}$.
\item[3.] The orbit closures are the intersections $S_J =S_{i_1} \cap \dots \cap S_{i_k}$ where $J=\{i_1,\ldots ,i_k\}$ is a subset of $\{1,\dots ,r\}$.  
\item[4.] $S_{J_1} \cap S_{J_2}= S_{J_1\cup J_2}$ and $\mathrm{codim}\,S_J = \sharp J$.
\end{enumerate}

\textit{Remark}. The integer $r$ is equal to the rank of $X$.

\medskip

We may ask how to define a set of equations of $\bar{X}$ in $\Gr( \dim \g^\sigma , \g)$: we do not know any reference to this question in the literature. In this paper, we give an answer when the rank of $X$ is equal to $l$. 
\newtheorem{thm}{Theorem}
\begin{thm} 
If the rank of $X$ is equal to $l$, $\bar{X}$ is defined in $\Gr(\frac{g-l}{2},\g)$ by linear equations.
\end{thm}
 
Let us give a sketch of the proof. We assume in this paper that $\rk X =l$. 

\newtheorem{defn}{Definition} \label{defmax}
\begin{defn} Let $W$ be a vector subspace of $\g$.
\begin{enumerate}
\item[(1)] The subspace $W$ is a nullspace for $(\g,w)$ if $w$ vanishes on $W\times W\times W$.
\item[(2)] The subspace $W$ is a maximal nullspace for $(\g,w)$ if it has maximal dimension for property (1).
\end{enumerate}
\end{defn}
We call $Y$ the set of all maximal nullspaces. This a closed subset of a grassmannian $\Gr(d',\g)$, where $d'$ is the dimension of maximal nullspaces for $(\g,w)$.
\\
For an involution $\sigma$ of $\g$, the direct sum $\g=\g^\sigma \oplus \g_{-1}$ where $\g_{-1}$ is the $-1$-eigenspace is orthogonal with respect to $\kappa$; moreover the subspace $\g_{-1}$ is a nullspace for $(\g,w)$. Any Borel subalgebra satisfies the condition (1) of Definition \ref{defmax}, so the maximal dimension is greater than or equal to $d:=\frac{g+l}{2}$. 

\medskip

We first prove that any maximal nullspace contains a Cartan subalgebra of $\g$, and we deduce from this fact that $d'=d$. If $W$ is a maximal nullspace which contains a Cartan subalgebra $\h$, let $\Phi$ be the root system of $(\g,\h)$. We prove that for any $\alpha\in\Phi$, the vector space $\co x_\alpha \oplus \co x_{-\alpha}$, generated by a root vector of $\pm \alpha$ meets $W$ along a line.  We deduce that the orbits of $Y$ under $G$ are the same as the parabolic subalgebras of $\g$ (the corresponding parabolic subalgebra of $W$ is $\mathfrak{p}=W+[W,W]$). The closed orbit consists of Borel subalgebras, and to prove the smoothness of $Y$, we analyze its tangent space over this orbit. This description corresponds to the wonderful compactification by the map $W\mapsto W^\bot$. We finish this paper with examples when $l=2$: $\mathfrak{sl}(3)$ and $\mathfrak{sp}(4)$.
\\
For classical simple Lie algebras, one knows a birational description of those wonderful compactifications. We summarize it in the next table.
\[
\begin{array}{c|c}
 \mathfrak{gl}(n) & \mathrm{Complete\, quadrics} \\
 \hline \\
 \mathfrak{sp}(2n) & \mathrm{Hilb}_2 \left(\mathrm{IG}(n,2n)\right) \\
 \hline \\
 \mathfrak{so}(2n) & \Gr(n,2n)/\sim \\
 \hline \\
 \mathfrak{so}(2n+1) & \Gr(n,2n+1)
\end{array}
\]  
For the second line, this is the Hilbert variety of two points on the isotropic grassmannian. For the third line, the equivalence $\sim $ identifies a subspace and its orthogonal. 

\section{Maximal nullspaces for $(\g,w)$} \label{section1}

We follow the above-mentioned sketch.

\newtheorem{prop}[thm]{Proposition} 
\begin{prop} \label{semsimple}
Every maximal nullspace contains a regular semisimple element.
\end{prop}

\textit{Remark}. Let $T$ be a maximal torus of $G$, and $\mu$ a one-parameter subgroup of $T$. We say that $\mu$ is \textit{regular} if all $\mu$-stable vector space $W$ is $T$-stable. In particular, if $\h$ is the Lie algebra of $T$, $W$ is $\h$-stable. See \cite{DKK} for more details. 

\medskip

Let $V$ be a maximal nullspace for $(\g,w)$ and recall that $\dim V=d' \geq d$. Take $\mu$ a regular one-parameter subgroup and let $V_0=\lim_{t\rightarrow 0} \mu(t)\cdot V$. The vector space $V_0$ is $\mu$-stable, so $\h$-stable, maximal for $(\g,w)$.

\newtheorem{lem}[thm]{Lemma}
\begin{lem} \label{semisimplelimit}
If $V_0$ contains a regular semisimple element, then so does $V$.
\end{lem}

\begin{proof}
We define the tautological vector bundle $K$ over  the grassmannian 
\\
$\Gr(d',\g)$, $p: K\rightarrow \Gr(d',\g)$ and $q: K \rightarrow \g$ the two projections. Let $\g_{rs}$ be the open set of regular semisimple elements of $\g$. Since $q$ is flat, $q\left(p^{-1}( \g_{rs})\right)$ is an open set of $\Gr(d',\g)$ containing $V_0$, and so there exists $t_0\in \co^*$ such that $\mu(t_0)V$ is included in $q\left(p^{-1}( \g_{rs})\right)$. Finally, $\mu(t_0)V$ contains a regular semisimple element, so does $V$.    
\end{proof}

We prove  Proposition \ref{semsimple} using a decreasing induction on 
\[
r=\sup_{\h} \dim V\cap \h,
\]
where $\h$ ranges through all Cartan subalgebras. 

\begin{proof}[Proof of Proposition \ref{semsimple}] 
\textbf{Initialization}. The case $r=l$ is obvious.

\textbf{Induction}. Let $r<l$, and assume the result is true for all $k$ such that $r<k\leq l$. Let $\h$ be a Cartan subalgebra such that $\dim V\cap\h =r$, $T$ be a maximal torus of $G$ such that $\h$ is the Lie algebra of $T$, $\mu$ be a regular one-parameter subgroup of $T$, and $\Phi$ be the root system of $(\g,\h)$.  It follows that $V_0=\lim_{t\rightarrow 0} \mu(t)\cdot V$ is $\h$-stable, so we can choose to write it as the direct sum
\[
V_0=V_0\cap \h \oplus \sum_{\alpha\in S} \co x_{\alpha},
\]
where $S$ is a subset of $\Phi$ and $x_{\alpha}$ a non zero vector of the root space $\g_{\alpha}$. Denoting $R=S\cap (-S)$, two cases appear.
\begin{enumerate}
\item[i)] $R=\varnothing$, so $\sharp S \leq \frac{g-l}2$, hence $l \geq \dim V_0\cap \h = \dim V_0 - \sharp S \geq l$, so this forces $\h \subset V_0$; the conclusion follows from Lemma \ref{semisimplelimit}.
\item[ii)] For $\alpha \in R$, the linear form $w(x_{\alpha},x_{-\alpha},.)$ vanishes on $V_0$, so we have $V_0\cap \h \subset \ker \alpha$. The vector space $V_0\cap \h \oplus \co (x_{\alpha}+x_{-\alpha})$ is an abelian Lie algebra consisting of semisimple elements so is contained in a Cartan subalgebra $\h_1$: $\dim V_0\cap \h_1 > \dim V_0\cap \h_0$. By induction, $V_0$ contains a regular semisimple element, hence so does $V$. \qedhere
\end{enumerate}
\end{proof}

\newtheorem{cor}[thm]{Corollary}
\begin{cor} \label{csgofmax}
\begin{enumerate}
\item[(a)] The maximal nullspace $V$ contains a Cartan subalgebra.
\item[(b)] There exists a one-parameter subgroup such that $V_0$ is a Borel subalgebra.
\item[(c)] $\dim V =d$.
\end{enumerate}
\end{cor}

\begin{proof}
Let $s$ be a regular semisimple element contained in $V$. 

(a) The centralizer $\mathfrak{c}(s)$ is a Cartan subalgebra. let $\overline{\g}$ be the quotient of $\g$ by $\mathfrak{c}(s)$, $\pi$ be the projection on $\overline{\g}$. Since $\psi_s=w(s,.,.)$ is a non degenerate skewsymmetric bilinear form over $\overline{\g}$ and $\pi(V)$ is an isotropic subspace,  
\begin{eqnarray*}
\dim \pi(V) &\leq & \frac{1}{2} \dim \overline{\g} \\
\dim V - \dim V\cap\mathfrak{c}(s) & \leq & \frac{1}{2}\left(\dim \g - \dim \mathfrak{c}(s) \right) \\
\dim V\cap\mathfrak{c}(s) &\geq& l
\end{eqnarray*}
and finally $\mathfrak{c}(s)\subset V$.

(b) Let $T$ be a maximal torus of $G$ with Lie algebra $\h:=\mathfrak{c}(s)$, $\Phi$ be the root system of $(\g,\h)$, $\mu$ be a regular one-parameter subgroup of $T$. It follows that the limit subspace $V_0$ has a decomposition 
\[
V_0=\h \oplus \sum_{\alpha\in S} \co x_{\alpha},
\]
where $S$ and $-S$ form a partition of $\Phi$. Now, for $\alpha$, $\beta \in S$  such that $\alpha +\beta $ is a root, $w(x_\alpha ,x_\beta ,x_{-\alpha -\beta })\neq 0$ proves that $\alpha +\beta\in S$, so we can choose a bases of $\Phi$ such that $S$ is the set of positive roots. 

(c) follows from (b) and $\dim V =\dim V_0$.   
\end{proof}

We can now describe the maximal nullspace $V$ by using a Cartan subalgebra $\h$ contained in $V$ and the associated root system $\Phi$:
\[
V=\h \oplus \bigoplus_{\alpha\in \Phi^+} L_{\alpha},
\]
where $L_{\alpha}$ is a vector subspace of dimension $1$ of $\g_{\alpha} \oplus \g_{-\alpha}$, $\g_\alpha$ the root space of $\alpha$.
 
Let $V$ be a maximal nullspace of $(\g,w)$ containing a Cartan subalgebra $\h$, $\Phi$ be the root system of $(\g,\h)$, choose $\alpha\in\Phi$, and let $h_0$ be an element of $\h$ such that its centralizer is $\mathfrak{c} (h)=\h \oplus \g^\alpha \oplus \g^{-\alpha}$. Using the argument in the proof of Corollary \ref{csgofmax} (the first point), we have $\dim V\cap \mathfrak{c} (h)\geq l+1$. But $\h \subset V$, so $\dim V\cap \left(\g^\alpha \oplus \g^{-\alpha} \right) \geq 1$. The linear form $w(x_\alpha , x_{-\alpha}, \cdot)$ is non zero on $\h$, and so $\dim V\cap \left(\g^\alpha \oplus \g^{-\alpha} \right) = 1$.

\begin{lem} \label{decomp}
Let $V\in Y$. There exists a Cartan subalgebra such that
\[
V=\h \oplus \bigoplus_{\alpha\in \Phi^+} L_{\alpha},
\]
where $\Phi$ is the root system of $(\g,\h)$ and $L_{\alpha}$ is a vector subspace of dimension $1$ of $\g_{\alpha} \oplus \g_{-\alpha}$.  
\end{lem}    

\textit{Remark}. If $\alpha$, $\beta$ are two positive roots such that $\alpha+\beta$ is a root, and if we denote by $v_{\alpha+\beta}$, $v_\alpha$, $v_\beta$  bases of $L_{\alpha+\beta}$, $L_{\alpha}$, $L_{\beta}$, then $w(v_\alpha,v_\beta,v_{\alpha+\beta})=0$ shows $v_{\alpha+\beta}$ is defined, up to a scalar, by $v_\alpha$, $v_\beta$. Let $\Delta$ be a root bases according to a Borel subalgebra. It is easy to compute that, up to conjugacy, we have two choices for $L_\alpha$, $\alpha\in \Delta$: this is a root space or not. 

\section{Orbits of $Y$} \label{section2}

The set $Y$ of maximal nullspaces of $(\g,w)$ is a closed set of $\Gr:=\Gr(d,\g)$, and is stable by the adjoint action of $G$. Thanks to Corollary \ref{csgofmax}, there is one closed orbit consisting of Borel subalgebras. In this section, we give a condition for two elements of $Y$ to be conjugate. 

\begin{prop} \label{parabV}
\begin{enumerate}
\item[(i)] The minimal parabolic subalgebra which contains $V\in Y$ is $\mathfrak{p}_V:=V+[V,V]$.
\item[(ii)] If $V_1$ and $V_2$ are two elements of $Y$ such that $\mathfrak{p}_{V_1}=\mathfrak{p}_{V_2}$, then $V_1$ and $V_2$ are conjugate under $G$.
\end{enumerate}
\end{prop}

\begin{proof}
(i) is obvious using Lemma \ref{decomp}.

(ii) Assume $\mathfrak{p}_{V_1}=\mathfrak{p}_{V_2}$. Up to conjugacy of $V_2$ under the adjoint group of $\mathfrak{p}_{V_1}$, assume the existence of a Cartan subalgebra $\h$ contained in $V_1\cap V_2$. Choosing a root system of $(\g,\h)$, there are two Borel subalgebras $\bor_1$ and $\bor_2$ such that, for $i\in\{1,2\}$
\[
\mathfrak{p}_{V_i}=\bor_i \oplus \bigoplus_{\alpha\in S_i} \co x_{-\alpha}^{(i)},  \quad
V_i=V_i\cap \bor_i \oplus \bigoplus_{\alpha\in S_i} \co (x_{\alpha}^{(i)}+x_{-\alpha}^{(i)}), 
\]
where $S_i$ is the set of positive roots (roots of $\bor_i$ such that $L_{\alpha}$ is not spanned by a root vector. There exists $g$ in the adjoint group of $\mathfrak{p}_{V_1}$ such that $g\cdot \bor_1 =\bor_2$. This forces $g\cdot x_\alpha^{(1)} =x_\beta^{(2)}$ with $\alpha\in S_1$ and $\beta\in S_2$. We conclude without difficulties.
\end{proof}

\textit{Remark}. The number of orbits in $Y$ is equal to $2^l$, the number of parabolic orbits. Indeed, Proposition \ref{parabV} says that $V_1$ and $V_2$ are in the same orbit in $Y$ if and only if $\mathfrak{p}_{V_1}$ and $\mathfrak{p}_{V_2}$ are conjugate. Conversely, for each parabolic subalgebra $\mathfrak{p}$, we can find an element of $Y$ such that $\mathfrak{p}_V=\mathfrak{p}$. 

\medskip

Moreover, there is only one orbit with dimension equal to $\dim Y$, given by the parabolic subalgebra $\g$,
\[
Y= \overline{G\cdot V},
\]
where $V$ is the direct sum $\h \oplus \sum_{\alpha \in \Phi^+} \co (x_{\alpha}+x_{-\alpha})$. 

Recall that, given a Cartan subalgebra $\h$ such that the restriction of the involution $\sigma$ to $\h$ is $-\mathrm{id}_\h$, it follows that $\sigma$ sends $x_{\alpha}$ to $t_\alpha x_{-\alpha}$ with $t_\alpha^2=1$. Since $\g^\sigma = \sum_{\alpha\in \Phi^+} \co (x_\alpha+t_\alpha x_{-\alpha})$, it is easy to see that 
\[
\g_{-1}=\left( \g^\sigma \right)^\bot = \h \oplus \sum_{\alpha \in \Phi^+} \co (x_{\alpha}-t_\alpha x_{-\alpha})
\]
(orthogonality being given by the Killing form). So, as sets, $Y$ and the wonderful compactification are isomorphic (we identify $\Gr(d-l,\g)$ and $\Gr(d,\g)$ by the isomorphism $W\mapsto W^\bot$).

As a consequence, $Y$ has dimension $d$. The next section shows that the equality is also true as a variety.  

\section{Equations of $Y$} \label{section4}

Recall that the grassmannian variety $\Gr$ has an exact sequence of locally free sheaves:
\begin{equation} \label{tauto}
0 \rightarrow K \rightarrow \g\otimes \Or_\Gr \rightarrow Q \rightarrow 0,
\end{equation}
where $K$ is the tautological sheaf of rank $d$ and $Q$ the quotient sheaf of rank $\frac{g-l}2 $. The datum $w\in \bigwedge^3 \g^\vee$ gives a section $w_1 :\Or_\Gr \rightarrow \bigwedge^3 K^\vee$ and by transposition, a morphism $ ^t w_1:\bigwedge^3 K \rightarrow \Or_\Gr $, whose the image is an ideal defining $Y$, denoted by $I_Y$.   

\medskip

\textit{Remark}. We describe this last morphism locally. Let $\Lambda\in Y$, take a base $x_1,\dots,x_d$ of $\Lambda$ and $y_1,\dots,y_{n-d}$ a base of a complementary $W$ of $\Lambda$. We can identify $\mathcal{U}=\mathrm{Hom}(\Lambda,W)$ with an affine open set of $\Gr$ by identifying $u\in \mathrm{Hom}(\Lambda,W)$ with the graph of $u$ viewed in $\Lambda\oplus W=\g$. Denote by $X_{i,j}$, with $1\leq i\leq d$ and $1\leq j \leq g-d$, the coordinate with respect to the previous bases. So $^t w_1 \, : \, \bigwedge^3 \Lambda \otimes \Or_{\Gr,\Lambda}\rightarrow \Or_{\Gr,\Lambda}$ sends  $x_{i_1}\bigwedge x_{i_2} \bigwedge x_{i_3}\otimes 1$ to  
\begin{equation} \label{eq1}
F_{i_1,i_2,i_3}:= w\left( x_{i_1}+\sum_j X_{i_1,j} \, y_j, \, x_{i_2}+\sum_j X_{i_2,j} \, y_j , \, x_{i_3}+\sum_j X_{i_3,j} \, y_j \right).
\end{equation}  
The polynomials $F_{i_1,i_2,i_3}$ for $1\leq i_1 < i_2 < i_3 \leq d$ span $I_{Y,\Lambda}$. To show that $Y$ is isomorphic, as a variety, to the wonderful compactification, we prove that $Y$ is a smooth variety. It is sufficient to show that $Y$ is smooth on the minimal orbit, so we must analyze the stalk of $\Omega_Y^1$, the sheaf of K\"ahler differentials, at a Borel subalgebra. 

\begin{thm} \label{smooth}
The variety $Y$ is smooth.
\end{thm}

Before proving the theorem, we need the following lemma.

\begin{lem} \label{corank}
Let $\bor$ be a Borel subalgebra of $\g$. The linear map $D$ defined by
\[
\begin{array}{rcl}
  \bigwedge^3 \bor & \rightarrow & \bor \otimes [\bor, \bor ] \\
  v_1\wedge v_2 \wedge v_3 & \mapsto & v_1\otimes [v_2,v_3] +v_2\otimes [v_3,v_1]+v_3\otimes [v_1,v_2] 
\end{array}       
\]
has corank less than or equal to $d$. 
\end{lem}

\begin{proof}
Let $\bor =\mathfrak{h} \oplus \bigoplus_{\alpha\in \Phi^+} \g_\alpha $ be a root space decomposition. For $h$ and $k$ in $\h$, $\alpha$ and $\beta$ in $\Phi^+$, we have
\begin{align} 
\alpha(k)h\otimes x_\alpha & = D(h\wedge k \wedge x_\alpha)+\alpha(h)k\otimes x_\alpha, \label{eq2} \\
\alpha(h)x_\alpha\otimes x_\beta  & = D(h\wedge x_\alpha \wedge x_\beta)+\beta(h)x_\beta\otimes x_\alpha - h\otimes [x_\alpha,x_\beta],  \label{eq3} \\
x_{\alpha+\beta}\otimes x_{\alpha+\beta} & = D(x_{\alpha+\beta}\wedge x_{\alpha} \wedge x_{\beta}) - x_\alpha\otimes [x_\beta,x_{\alpha+\beta}] + x_\beta\otimes [x_\alpha,x_{\alpha+\beta}]. \label{eq4}
\end{align}
Let $W$ be the subspace of $\mathrm{coker}\; D$ spanned by $h_\alpha \otimes x_\alpha$, where $\alpha(h_\alpha)=2$ and $\alpha\in \Phi^+$. For suitable $h$ and $k$, equalities \eqref{eq2} and \eqref{eq3} show that $h\otimes x_\alpha$ with $\alpha(h)=0$ are in $\mathrm{Im}\, D$, and $x_\alpha\otimes x_\beta$ with $\alpha\neq \beta$ are in $W$, hence it follows from \eqref{eq4} that $x_\alpha \otimes x_\alpha \in W$ if $\alpha$ is not simple. Finally, $\coker D \subset W$. So the number of generators is $\frac{g-l}{2} + l=d$. 
\end{proof}

\begin{proof}[Proof of Theorem \ref{smooth}.]
Let $\bor$ be a Borel subalgebra, $\h\subset \mathfrak{b}$ be a Cartan subalgebra, $\Phi$ be the root system of $(\g,\h)$, with positive roots given by $\mathfrak{b}$, and $\g= \bor \oplus \mathfrak{n}^-$ be a root space decomposition with bases $x_1,\dots, x_d$ for $\bor$ (positive root vectors and a bases of $\h$), $y_1,\dots,y_{n-d}$ for $\mathfrak{n}^-$ (negative root vectors) such that $\kappa(x_i,y_i)\neq 0$, for $i\in \{1,\dots,d\}$. We use the following exact sequence on sheaves of differentials:
\[
I_Y / I_Y^2 \rightarrow \Omega_\Gr \otimes \Or_Y \rightarrow \Omega_Y \rightarrow 0.
\] 
Locally, we can compute the differential of $F_{i_1,i_2,i_3}$ in $\Omega_{\Gr,\bor}$ (image of the first map in the sequence). The result is 
\begin{multline} \label{diff}
dF_{i_1,i_2,i_3}=\sum_j w(y_j,x_{i_2},x_{i_3}) \, dX_{i_1,j} \, + \, \sum_j w(x_{i_1},y_j,x_{i_3}) \, dX_{i_2,j} \\ + \, \sum_j w(x_{i_1},x_{i_2},y_j) dX_{i_3,j}.
\end{multline}
\\
But $\Omega_{\Gr,\bor}$ is isomorphic to $(\g/\bor)^\vee \otimes \bor$, sending $dX_{i,j}$ to $y_i^\vee \otimes x_j$, where $y_i^\vee =\frac{1}{\kappa(x_i,y_i)}\kappa(x_i,\cdot)$ (the duality is relating to the Killing form $\kappa$, and we normalize to have $y_i^\vee (y_i)=1$). For the first sum in \eqref{diff}, we have 
\begin{eqnarray}
\sum_j w(y_j,x_{i_2},x_{i_3}) \, y_j^\vee \otimes x_{i_1} &=& \kappa \left( \sum_j  \frac{\kappa(y_j,[x_{i_2},x_{i_3}])}{\kappa(y_j,x_j)} x_{i_1} \, , \, \cdot \right) \otimes x_{i_1} \\
&=& \kappa([x_{i_2},x_{i_3}],\cdot)\otimes x_{i_1} .
\end{eqnarray}

The composition map $\bigwedge^3\mathfrak{b} \otimes \Or_{\Gr,\bor} \rightarrow I_{Y,\bor} / I_{Y,\bor}^2 \rightarrow \Omega_{\Gr,\bor} \otimes \Or_{Y,\bor}$ sends 
\\
$x_{i_1}\bigwedge x_{i_2}\bigwedge x_{i_3}$ to 
\[
\kappa([x_{i_2},x_{i_3}],\cdot)\otimes x_{i_1}+\kappa([x_{i_1},x_{i_2}],\cdot)\otimes x_{i_3}+\kappa([x_{i_3},x_{i_1}],\cdot)\otimes x_{i_2}. 
\]   
Thanks to Lemma \ref{corank}, we conclude that corank of $I_{Y,\bor} / I_{Y,\bor}^2 \rightarrow \Omega_{\Gr,\bor} \otimes \Or_{Y,\bor}$ is less than of equal to $d$, so $\rk \Omega_{Y,\bor} \leq d$, and the result follows.  
\end{proof}

A consequence of Theorem \ref{smooth} is that $Y$ is isomorphic to the wonderful compactification. The next theorem shows that equations of the wonderful compactification in the grassmannian are linear.

\begin{thm} \label{lineareq}
The equations of $Y$ in $\Gr$ are linear.
\end{thm}
\begin{proof}
Recall that $\bigwedge^{d} K = \Or_\Gr (-1)$, so $\mathrm{Hom}\left(\bigwedge^{d} K, \bigwedge^{3} K \right) \simeq \bigwedge^{d-3} K^\vee$, and so $\bigwedge^{3} K(1)\simeq \bigwedge^{d-3} K^\vee$. Moreover, from \eqref{tauto}, we have $\bigwedge^{d-3}\g^\vee \otimes \Or_\Gr \twoheadrightarrow \bigwedge^{d-3} K^\vee$. This forces $\bigwedge^{d-3} K^\vee$ to be spanned by its sections, and so does it to $\bigwedge^{3} K(1)$. Thanks to the morphism $^t w_1$, $I_{Y}(1)$ is spanned by its sections.   
\end{proof}

We give a result on global sections of $I_{Y}(1)$ when $\g$ is a simple Lie algebra. Extending $w \, : \, \bigwedge^3\g \rightarrow \co$ to $\bigwedge^{k+3} \g \rightarrow \bigwedge^k \g$ with $k$ a positive integer, we build a $\g$-invariant differential operator on $\bigwedge \g $, denoted by $\delta^* $, satisfying $(\delta^*)^2=0$. On the other side, by identifying $\g$ and his dual by the Killing form, $w$ can be seen as an element of $\bigwedge^3 \g$, the morphism of $\g$-module $\bigwedge^{k} \g \stackrel{\wedge w}{\rightarrow} \bigwedge^{k+3} \g$ with $k$ a non negative integer defines another $\g$-invariant differential operator on $\bigwedge \g$, denoted by $\delta$.   
\\
Let $\h$ be a Cartan subalgebra of $\g$, $\Phi$ be the root system of $(\g,\h)$. We choose a base of the root system $\Phi$, and denote by $\Gamma_{2\rho}$ the irreducible representation of highest weight $2\rho$, where $\rho$ is the half sum of positive roots. The module $\bigwedge \h \otimes \Gamma_{2\rho}$ represents all occurrences of $\Gamma_{2\rho}$ in $\bigwedge \g$. Thus, $\Gamma_{2\rho}$ appears only in $\bigwedge^k \g$, when $k$ ranges $\{g-d,g-d+1,\ldots,d-1,d\}$. For $k$ a non negative integer, let $\overline{\wedge}^k \g$ be a $\g$-submodule of $\bigwedge^k \g$ such that $\bigwedge^k \g=\overline{\wedge}^k \g \oplus \bigwedge^{k-n+d} \h \otimes \Gamma_{2\rho}$ if $k\in\{g-d,g-d+1,\ldots,d-1,d\}$ and $\overline{\wedge}^k \g:=\bigwedge^k \g$ otherwise. Thanks to the fact that $\delta$ and $\delta^*$ preverse the weights, we can defined the restriction of $\delta$, $\delta^*$ to $\overline{\wedge} \g$. Moreover the restrictions of $\delta$ and $\delta^{*}$ to $\bigwedge \h \otimes \Gamma_{2\rho}$ are trivial.

\begin{lem}\label{exactseq}
Assume $\g$ is simple. The sequences  
\[
0 \rightarrow \overline{\wedge}^{m'} \g \stackrel{\delta}{\longrightarrow} \overline{\wedge}^{m'+3} \g \stackrel{\delta}{\longrightarrow} \cdots \stackrel{\delta}{\longrightarrow} \overline{\wedge}^{g-3} \g \stackrel{\delta}{\longrightarrow} \bigwedge^{g} \g \rightarrow 0,
\]
\[
0 \rightarrow \co \stackrel{\delta}{\longrightarrow} \overline{\wedge}^{3} \g \stackrel{\delta}{\longrightarrow} \cdots \stackrel{\delta}{\longrightarrow} \overline{\wedge}^{m-3} \g \stackrel{\delta}{\longrightarrow} \overline{\wedge}^{m} \g \rightarrow 0,
\]
with $m\in\{g-2,g-1,g\} $ and $m' \in \{ 0,1,2\}$, are exact.
\end{lem}

\textit{Remarks.} We could write sequences with $\delta^*$ decreasing wedge power of $\g$, which gives other exact sequences for $\overline{\wedge} \g$.

The complex $(\bigwedge \g, \delta)$ is a direct sum of two complex, the first $\overline{\wedge} \g $ is acyclic, and the second given by $\bigwedge \h \otimes \Gamma_{2\rho}$ with $\h$ a Cartan subalegra of $\g$, is trivial.

\medskip

Assume for the moment this lemma. In the proof of Theorem \ref{lineareq}, $ ^t w_1:\bigwedge^3 K(1) \rightarrow \Or_\Gr (1)\simeq \bigwedge^{d} K $ gives a $\g$-invariant morphism on global sections $\bigwedge^{d-3} \g = \mathrm{H}^0 (\bigwedge^3 K(1)) \rightarrow \mathrm{H}^0 (\Or_\Gr(1))=\bigwedge^{d} \g $, it is just $\delta$. 

\begin{prop} \label{inclsect}
If $\g$ is simple, then $\mathrm{H}^0 (I_Y(1))$ contains $\delta \left(\bigwedge^{d-3} \g \right)$. 
\end{prop}

\begin{proof}
The image of $ ^t w_1 : \bigwedge^3 K(1) \rightarrow \Or_\Gr (1)$ is $I_Y(1)$.
\end{proof}

This proposition shows that we can embed the wonderful compactification in a projective space with dimension smaller than $\mathbb{P}(\bigwedge^d \g)$. 

\medskip

Now we prove Lemma \ref{exactseq}. The main idea is the study of $\zeta=\delta^* \delta + \delta \delta^*$ as a $\g$-invariant differential operator. The multiplication by an element of $\g$ and the derivation (action by an element of $\g^\vee$) spans the ring of differential operators on $\bigwedge \g $ identified to the Clifford algebra $\Clif$ where $ev$ is the duality bracket. Recall that $\Clif$ has a $\mathbb{Z}/2\mathbb{Z}$-graduation, which allows us to put a structure of Lie superalgebra. 
\\
Define a filtration $(F^i)$ with $F^i$ spanned by products of multiplications and at most $i$ derivations. We recall two useful results:
\begin{enumerate}[1.]
\item  $[F^i,F^j]\subset F^{i+j-1}$,
\item an element $\chi$ of $F^i$ is null if $\chi_{|_{\wedge^k \g}}=0$; for $k\leq i$, in other words, elements of $F^i$ are completly known by the image of $\bigoplus_{k\leq i} \bigwedge^k \g $. 
\end{enumerate}

For our case, $\delta \in F^0$ and $\delta^* \in F^3$ so $\zeta=[\delta,\delta^*] \in F^2$. The Casimir operator $c$ and powers of Euler operator $e$, $e^0=\mathrm{id}$, $e$, $e^2$ ($e$ is defined as $e_{|_{\wedge^i \g}}= i\cdot\mathrm{id}$) are $\g$-invariant differential operators in $F^2$. We need the following lemma to prove that $\zeta$ is a linear combination of $c$, id, $e$ and $e^2$.

\begin{lem} \label{wedge2}  
Denote by $\Gamma_{a_1,\ldots,a_l}$ the irreducible representation of $\g$ with highest weight $a_1 \omega_1 + \cdots + a_l \omega_l$, where $\omega_1, \ldots, \omega_l$ are the fundamental weights. We have:
\begin{enumerate}[(i)]
\item $\bigwedge^2 \mathfrak{sp}(2n) = \mathfrak{sl}(n+1) \oplus \Gamma_{2,0,\ldots,0,1,0} \oplus \Gamma_{0,1,0,\ldots,0,2}$, with $n\geq 3$, and $\bigwedge^2 \mathfrak{sl}(3) = \mathfrak{sl}(3) \oplus \Gamma_{3,0} \oplus \Gamma_{0,3}$.
\item $\bigwedge^2 \mathfrak{sp}(2n) = \mathfrak{sp}(2n) \oplus \Gamma_{2,1,0,\ldots,0}$, for $n\geq 2$.
\item $\bigwedge^2 \mathfrak{so}(n) = \mathfrak{so}(n) \oplus \Gamma_{1,0,1,0,\ldots,0} $, for $n\geq 6$.
\item $\bigwedge^2 \mathfrak{f}_4 = \mathfrak{f}_4 \oplus \Gamma_{0,1,0,0}$.
\item $\bigwedge^2 \mathfrak{g}_2 = \mathfrak{g}_2 \oplus \Gamma_{3,0}$.
\item $\bigwedge^2 \mathfrak{e}_6 = \mathfrak{e}_6 \oplus \Gamma_{0,0,0,1,0,0}$.
\item $\bigwedge^2 \mathfrak{e}_7 = \mathfrak{e}_7 \oplus \Gamma_{0,0,1,0,0,0,0}$.
\item $\bigwedge^2 \mathfrak{e}_8 = \mathfrak{e}_8 \oplus \Gamma_{0,0,0,0,0,0,0,1}$. \qedhere
\end{enumerate}
\end{lem}

\begin{proof}
We treat only $\mathfrak{sl}(n+1)$, $\mathfrak{sp}(2n)$ and $\mathfrak{so}(n)$. Other cases are given by computation with a program named \textit{LIE}.
\begin{enumerate}[(i)]
\item Write $V=\co^{n+1}$. We have $V \otimes V^\vee=\mathfrak{sl}(n+1) \oplus \co$, so
\begin{align*}
\bigwedge^2 \left(V \otimes V^\vee  \right) &= \bigwedge^2 V \otimes \mathrm{Sym}^2 V^\vee  \oplus \mathrm{Sym}^2 V \otimes \bigwedge^2 V^\vee \\
\bigwedge^2 \mathfrak{sl}(n+1) \oplus \mathfrak{sl}(n+1) &= \Gamma_{0,1,0,\ldots,0} \otimes \Gamma_{0,\ldots,0,2} \oplus \Gamma_{0,\ldots,0,1,0} \otimes \Gamma_{2,0,\ldots,0}.
\end{align*}
The result follows from Proposition 15.25 in \cite{FH}.
\item  As a representation of $GL(\co^{2n})$, $\bigwedge^2 \mathfrak{sp}(2n)=\bigwedge^2 \mathrm{Sym}^2(\co^{2n})$ is irreducible of partition $(3,1)$, so using branching rules in \cite{FH}, we obtain the wished formula.
\item The $GL(\co^{n})$-module $\bigwedge^2 \mathfrak{so}(n)$ is irreducible of partition $(2,1,1)$; we conclude with branching rules.\qedhere
\end{enumerate}
\end{proof}

Except for $\mathfrak{sl}(n)$, the $\g$-module $\co\oplus \g \oplus \bigwedge^2 \g$ has four irreducible factors: $c$, id, $e$, $e^2$ form a bases of $\g$-invariant differential operators of $F^2$, so $\zeta$ is a linear combination of $c$, id, $e$ and $e^2$. 
\\
For $\mathfrak{sl}(n)$, $n\geq 3$, remark that $\bigwedge^2 \mathfrak{sl}(n)=\mathfrak{sl}(n)\oplus W \oplus W^\vee$, with $W=\Gamma_{2,0,\ldots,0,1,0}$ or $W=\Gamma_{3,0}$, and $\zeta$, $c$, id, $e$, $e^2$ do not distinguish an irreducible representation and its dual. Considering $W\oplus W^\vee$ as one factor, $\co\oplus \g \oplus \bigwedge^2 \g$ has four factors. We can treat $\mathfrak{sl}(n)$ as other simple Lie algebras.  

\medskip

There exist a scalar $\alpha$ and a polynomial $P$ of degree less than or equal to $2$ such that $\zeta-\alpha c =P(e)$. Applying this expression on $1\in \co$ and $w\in \bigwedge^3 \g$, it follows that $P(0)=P(3)=\delta^*(w)$. But the isomorphism $\bigwedge^k \g \simeq \bigwedge^{n-k} \g$ shows that $P(n-3)=P(n)=P(3)$. Finally, $P$ is constant, thus
\begin{equation} \label{linzeta}
\zeta= \alpha c + \delta^*(w)\mathrm{id}. 
\end{equation}
If $\Gamma_\lambda$ is an irreducible representation of highest weight $\lambda$, denote by the scalar $c_\lambda$ the action of $c$ on $\Gamma_\lambda$. So, applying \eqref{linzeta} to the highest weight vector of $\Gamma_{2\rho}$, we have $0=\alpha c_{2\rho} + \delta^*(w)$, and so 
\[
\zeta=\delta^*(w)\left(\mathrm{id}-\frac{1}{c_{2\rho}}c\right). 
\] 
A Kostant argument says that $c_\lambda < c_{2\rho}$, if $2\rho$ dominates the dominant weight $\lambda$. Moreover, an irreducible $\g$-module $\Gamma$ which appears in $\bigwedge \g$ has highest weight dominated by $2\rho$: the restriction of $\zeta$ to $\Gamma$ is just the multiplication by a non-zero scalar, except for $\Gamma_{2\rho}$.
 
\begin{proof}[Proof of Lemma \ref{exactseq}]
Let $k$ be a positive integer, and $\Gamma$ be an irreducible representation which appears in Ker$(\delta)\cap \overline{\wedge}^k \g$, $\zeta_{|_{\Gamma}}=\lambda \mathrm{id}_\Gamma$ with $\lambda\neq 0$. This statement implies that if $x\in \Gamma $ then $x=\delta \left( \frac{1}{\lambda} \delta^* (x)\right) \in \mathrm{Im}(\delta)$. So all irreducible representations appearing in $\mathrm;{Ker}(\delta)\cap \overline{\wedge}^k \g$ are subsets of $\mathrm{Im}(\delta)\cap \overline{\wedge}^k \g$, so $\mathrm{Ker}(\delta)\cap \overline{\wedge}^k \g=\mathrm {Im}(\delta)\cap \overline{\wedge}^k \g$. This forces the two sequences to be exact. 
\end{proof}

\section{Correspondence between orbits of $Y$ and sets of simple roots}  \label{section3bis}

The results on orbits of $Y$ agree with the nice properties of the wonderful compactification. Let $\h$ be a Cartan subalgebra of $\g$ such that $\sigma_{|_\h}=-\mathrm{id}_\h$. It follows that $\sigma(\Phi)=-\Phi$ where $\Phi$ is the root system of $(\g,\h)$.
  
\begin{enumerate} 
\item[1.] Denote by $\mathfrak{p}_i$ the parabolic subalgebra spanned by a Borel subalgebra $\bor$ such that $\h \subset \bor$ and by $\g_{-\alpha_j}$ where $j$ runs over $\{\alpha_1,\dots,\alpha_{i-1},\alpha_{i+1},\dots,\alpha_{l}\}$, with $\Delta=\{\alpha_1,\dots,\alpha_{l}\}$ a root bases with respect to $\bor$.
The closure orbits corresponding to $\mathfrak{p}_i$ are the $l$ hypersurfaces $S_{\alpha_i}$.    
\item[2.] The closure orbits are $S_I =\cap_{\alpha \in I} S_\alpha$ where $I\subset \Delta$.
\end{enumerate} 

Now, we explain the correspondence between orbit closures and subsets of $\Delta$. Denote by $\bor^-$ the Borel subalgebra spanned by $\h$ and $x_{-\alpha}$ (with $\alpha\in\Delta$), $\mathfrak{n}=[\bor^- , \bor^- ]$, $N$ the adjoint group of $\mathfrak{n}$, and $T$ the maximal torus of $G$ with Lie algebra equal to $\h$. Let
\[
Y_\h =\{ U\in Y \; \mathrm{such \, that} \ U\cap \bor^- = \h \}.
\]   
The last remark of Section \ref{section1} shows that 
\[
\begin{array}{cccccc}
p \, : & \co^\Delta & \rightarrow & Y_\h \\
       & (t_\alpha)_{\alpha \in \Delta} & \mapsto & \h \oplus \bigoplus_{\alpha\in \Phi^+} (x_{\alpha}+t_\alpha x_{-\alpha})
\end{array}       
\] 
is a $T$-equivariant isomorphism, so $\dim Y_\h =l$ (the action of $T$ on $\co^\Delta$ is defined by $e^h.(t_\alpha)_{\alpha \in \Delta}=(e^{\alpha(h)}t_\alpha)_{\alpha \in \Delta}$). We prove that there is a correspondence between $G$-orbit closures of $Y$ and $T$-orbit closures of $\co^\Delta$, which are the $\co^I$, where $I$ is a subset of $\Delta$ (in fact, it is easy to show that $T$-orbits have the form $(\co^{*})^I$ with $I$ a subset of $\Delta$).   

\begin{prop} \label{immersion}
The morphism 
\[
\begin{array}{cccc}
\psi \, : & N\times Y_\h & \rightarrow & Y \\
       & (n,U) & \mapsto & n.U
\end{array}       
\] 
is an open immersion. 
\end{prop}

\begin{proof}
If $n_1$, $n_2$ are in $N$, and $U_1$, $U_2$ are two elements of $Y_\h$ such that $n_1.U_1=n_2.U_2$, then $n_2^{-1}n_1.\h = (n_2^{-1}n_1.U_1)\cap\bor^- =U_2\cap \bor^- =\h$, that is to say $n_2^{-1}n_1 =1$ (the normalizer of a maximal torus contains no unipotent elements), so $U_1=U_2$: $\psi$ is injective. Moreover, $\dim N\times Y_\h =\dim Y$ forces $\psi$ to be dominant, so finally $\psi$ is birational. We use a corollary of the main theorem of Zariski: since $\psi$ is birational, with finite fibres, then, because $Y$ is smooth, $\psi$ is an isomorphism between $X$ and an open subset $U$ of $Y$.
\end{proof}

Let $O$ be an orbit of $Y$ and $U\in O$. Recall that $\mathfrak{p}_U=U+[U,U]$ is a parabolic subalgebra of $\g$, so there exists a subset $S$ of $\Phi^+$ such that $\mathfrak{p}_U$ is conjugate to 
$\mathfrak{p}=\bor \oplus \sum_{\alpha\in -S} \g^\alpha$. Now we build an element $V$ of $Y_\h$ such that $\mathfrak{p}=V+[V,V]$:
\[
V=\h \oplus \bigoplus_{\alpha\in \Phi^+ \smallsetminus S} \co x_\alpha \oplus \bigoplus_{\alpha\in  S} \co (x_\alpha+x_{-\alpha}).
\]
Since $\mathfrak{p}_U$ and $\mathfrak{p}$ are conjugate, Proposition \ref{parabV} implies that $V$ and $U$ are conjugate, so $O\cap Y_\h \neq \emptyset $.

\begin{prop} \label{TGorb}
There is a bijection between the $T$-orbit closures of $Y_\h$ and the $G$-orbit closures of $Y$, defined as follows:
\[
\begin{array}{ccc}
\left\{ G-\mathrm{orbit\, closure\, of\,} Y \right\} & \longleftrightarrow & \left\{ T-\mathrm{orbit\, closure\,  of\,} Y_\h \right\} \\
 \overline{O} & \mapsto & \overline{O}\cap Y_\h.
\end{array} 
 \]
This map preserves intersection. Moreover, $G$-orbit closures are smooth.
\end{prop}

\textit{Remark}. For a $G$-orbit closure $\overline{O}$, the set $\overline{O}\cap Y_\h$ is a $T$-stable closed set, so is isomorphic to $(\co)^I$ where $I$ is a subset of $\Delta$. But it is not impossible that a $G$-orbit $O$ meets several $T$-orbits of $Y_\h$, so the correspondence of Proposition \ref{TGorb} fails for orbits.

\begin{proof}
Let $\overline{O}$ be an $G$-orbit closure of $Y$. The closed set $\overline{O}\cap \psi(N\times Y_\h)$ of $\psi(N\times Y_\h)$ is $T$-stable and $N$-stable, so there exists a subset $I$ of $\Delta$ such that $\overline{O}\cap \psi(N\times Y_\h)=\psi(N\times \co^I)$. But $\overline{O}\cap \psi(N\times Y_\h)$ is open in $\overline{O}$, so $\overline{O}=\overline{\psi(N\times \co^I)}$. This forces the map $\overline{O}\mapsto \overline{O}\cap Y_\h $ to be an injection. Since $Y$ has $2^l$ orbit closures and the cardinal of $\mathcal{P}(\Delta)$ is equal to $2^l$, the bijection follows. 

\medskip

It is clear that $\overline{O}\cap \psi(N\times Y_\h)\simeq N\times \co^{I}$ is smooth. If the singular set of $\overline{O}$ is non empty, it is a $G$-stable closed set, so meets $\overline{O}\cap \psi(N\times Y_\h)$, which forces the smoothness of $\overline{O}$. 
\end{proof}

\textit{Remarks}. Let $\overline{O}$ be a $G$-orbit closure of $Y$. There exists $I\in \mathcal{P}(\Delta)$ such that $\overline{O}=\overline{\psi(N\times \co^I)}$. 
\begin{enumerate}
\item[(a)] We have codim $\bar{O}_{I}=\sharp (\Delta \smallsetminus I)$.
\item[(b)] For $\alpha\in \Delta$, we define $S_\alpha =\overline{\psi(N\times \co^{\Delta\smallsetminus \{\alpha\}})}$. All closed sets $S_\alpha$ are of codimension one, and the family $\{S_\alpha \, , \, \alpha\in\Delta \}$ meets transversally. 
\end{enumerate}

\section{Examples}

\subsection{The case $\mathfrak{sl}(3)$}

Let $V$ be a vector space of dimension $3$ and $S^2 V$ be the vector space of conics on $V$. The closure $Z$ of the graph of the duality isomorphism in $\mathbb{P}(S^2 V)\times \mathbb{P}(S^2 V^\vee)$, defined as,
\[
\begin{array}{ccc}
\mathbb{P}(S^2 V) & \cdots \rightarrow &  \mathbb{P}(S^2 V^\vee ) \\
  q & \mapsto &  \wedge^2 q, 
\end{array} 
\]
is called the variety of \textit{complete conics}, and the map $p:Z\rightarrow \mathbb{P}(S^2 V)$ is known to be the blowing up of $\mathbb{P}(S^2 V)$ along the Veronese surface (conics of rank one on $V$). We refer to the appendix of \cite{TH} for more results. So, if $J$ is the sheaf of ideals of the Veronese surface in $\mathbb{P}(S^2(V))$, $Z=\mathrm{Proj}(J)$ is embedded in $\mathrm{Proj} \left(\mathrm{H}^0(J(3))\right )=\mathbb{P}(\co\oplus \Gamma_{2\rho})$, where $\Gamma_{2\rho}$ is the irreducible $\mathfrak{sl}(3)$-module of dimension $27$ which corresponds to the irreducible representation of highest weight $2\rho$. We finish our description with the commutative diagram:
\[
\begin{array}{ccc}
Z=P\cap S & \rightarrow & P:=\mathbb{P}(\co\oplus \Gamma_{2\rho}) \\
\downarrow & & \downarrow \\
S:=\mathbb{P}(S^2 V)\times \mathbb{P}(S^2 V^\vee) & \rightarrow & \mathbb{P}(S^2 V\otimes S^2 V^\vee)
\end{array}
\]
\textit{Remark}. The variety $Z$ can be defined as:
\[
Z=\{ ([q\otimes q'])\in \mathbb{P}(S^2 V\otimes S^2 V^\vee) \, \mathrm{such \, that} \ qq' \in \co \mathrm{Id} \}.
\]
We can find easly orbit closures: one when $\rk q =1$, an other when $\rk q' =1$, and their intersection (the closed orbit).

\medskip 

Littelmann and Procesi show that $Z$ is isomorphic to the wonderful compactification of $PGL(3)/PSO(3)$. In this part, we find equations defining the wonderful compactification in $\mathbb{G}(3,\mathfrak{sl}(3))$.

Let $e_1,e_2,e_3$ be a bases of $V$ and consider the quadratic form $q=e_1^2+e_2^2+e_3^2$. The morphism $\sigma:PGL(3)\rightarrow PGL(3)$ which sends $[g]$ to $[q^{-1}\, ^tg^{-1} \, q]$ is an involution, and $PGL(3)^\sigma = PSO(q)$. 
\\
We have seen in Section \ref{section4} that the sheaf of equations of $Y$ in $\mathbb{G}(5,\sl)$ (and so $\bar{X}$ in $\mathbb{G}(3, \sl )$) is the image $I_{\bar{X}}$ of $\bigwedge^3 K \rightarrow \mathcal{O}_{\mathbb{G}(5,\mathfrak{sl}(3))}$.
 So, thanks to Proposition \ref{inclsect}, for $\sl $, $\mathrm{H}^0(I_{\bar{X}}(1))$ is a submodule of the $\mathfrak{sl}(3)$-module $\bigwedge^3\mathfrak{sl}(3)$ which is isomorphic to $\bigwedge^2 \mathfrak{sl}(3)\oplus \co \oplus \Gamma_{2\rho}$, and contains $\bigwedge^2 \mathfrak{sl}(3)$. But two cases are impossible:
\begin{enumerate}[i.]
\item If $\mathrm{H}^0(I_{\bar{X}}(3))=\bigwedge^2 \mathfrak{sl}(3)\oplus \Gamma_{2\rho}$, then $\bar{X}$ satisfies the equations of $\mathbb{P}(\co)$, and so $\bar{X}$ is a point.   
\item If $\mathrm{H}^0(I_{\bar{X}}(3))=\bigwedge^2 \mathfrak{sl}(3)\oplus \co)$, then $\bar{X}\subset \mathbb{P}(\Gamma_{2\rho})$. But elements which are not in the closed orbit of $\bar{X}\subset \mathbb{P}(\bigwedge^3 \g)$ have factors in $\bigwedge^{3} \g$ whose images by $\delta$ are not equal to zero (we refer to Lemma \ref{decomp} to explain it).
\end{enumerate}

Hence, $\mathrm{H}^0(I_{\bar{X}}(3))=\bigwedge^2 \mathfrak{sl}(3)$. In particular, $\bar{X}$ satisfies the equations of $\mathbb{P}(\co \oplus \Gamma_{2\rho}) \subset \mathbb{P}(\bigwedge^3 \mathfrak{sl}(3))$. We summarize the results obtained so far. 

\begin{lem}
The wonderful compactification is the intersection of $\mathbb{P}(\co\oplus \Gamma_{2\rho})$ and $\Gr(3,\g)$ in $\mathbb{P}(\bigwedge^3 \mathfrak{sl}(3))$.
\end{lem}

Since $Z$ and $\bar{X}$ are identified to subvarieties of $\mathbb{P}(\co\oplus \Gamma_{2\rho})$, we can prove that the wonderful compactification is isomorphic to the variety of complete conics.

\begin{prop}
There exists a $PGL(3)$-equivariant automorphism of $\mathbb{P}(\Gamma_{2\rho}\oplus \co)$  which sends $\bar{X}$ to $Z$.
\end{prop} 

\begin{proof} 
It is enough to find a $PGL(3)$-invariant isomorphism of $\mathbb{P}(\co\oplus \Gamma_{2\rho})$ which sends an element of the biggest orbit of $\bar{X}$ to an element of the biggest orbit of $Z$. 

Now, $\mathfrak{sl}(3)$ being viewed as a submodule of $V\otimes V^\vee$, the composition of morphisms of $\mathfrak{sl}(3)$-modules denoted by $\Psi$,
\[ 
\bigwedge^3 \mathfrak{sl}(3) \rightarrow  V \otimes V\otimes V\otimes V^\vee \otimes V^\vee \otimes V^\vee \rightarrow V\otimes V \otimes V^\vee \otimes V^\vee \rightarrow S^2 V\otimes S^2 V^\vee,
\]
has a restriction to $\co\oplus \Gamma_{2\rho} \rightarrow \co\oplus \Gamma_{2\rho}$ which is an isomorphism.

As $q=e_1^2+e_2^2+e_3^2$ is a non-degenerate conic on $V$, it gives $q\otimes q^\vee =(e_1^2+e_2^2+e_3^2)\otimes ((e_1^\vee)^2+(e_2^\vee)^2+(e_3^\vee)^2)$ a point of the open orbit of $Z$, and $\mathfrak{so}(q)$, a point of the open orbit of $\bar{X}$, seen in $\bigwedge^3 \mathfrak{sl}(3)$ as $(e_1\otimes e_2^\vee-e_2\otimes e_1^\vee)\wedge(e_2\otimes e_3^\vee-e_3\otimes e_2^\vee)\wedge(e_1\otimes e_3^\vee-e_3\otimes e_1^\vee)$. The last morphism sends the point of $\bar{X}$ to $q\otimes  q^\vee + z^2$, $z=e_1\otimes e_1^\vee+e_2\otimes e_2^\vee+e_3\otimes e_3^\vee$. Moreover, the component of $q\otimes q^\vee$ on the factor $\co$ in $\Gamma_{2\rho}\oplus \co$ is $1/2 z^2$, so $\phi=\mathrm{id}_{\Gamma_{27}}+\frac{1}{3} \mathrm{id}_{\co}$, the automorphism of $\Gamma_{2\rho}\oplus \co$ sends $q\otimes  q^\vee + z^2$ to $q\otimes  q^\vee$. Finally, the composition map $\phi \, \circ \, \Psi$ is a $PGL(3)$-equivariant automorphism of $\mathbb{P}(\Gamma_{2\rho}\oplus \co)$  which sends $\bar{X}$ to $Z$.
\end{proof} 

\textit{Remark}.  In fact, we have the following resolution: 
\begin{multline*}
0\rightarrow \mathcal{O}_{\mathbb{G}(2,\mathfrak{sl}(3))}(-6)=\bigwedge^{10} \bigwedge^3 K \rightarrow \bigwedge^{9}\bigwedge^3 K \rightarrow \cdots \rightarrow \bigwedge^{2}\bigwedge^3 K  \\ \rightarrow \bigwedge^3 K \rightarrow I_{\bar{X}} \rightarrow 0.
\end{multline*}
Since $\mathrm{codim}\, \bar{X}=\rk \bigwedge^3 K=10$, the resolution is exact. This sequence could help us to compute with Schur functors the cone of the wonderful compactification.

\subsection{The case $\Sp $}

Let $V$ be the irreducible representation of $\Sp$ of dimension $4$. We describe the elements of $\Sp$ in block form in the decomposition $V=U\oplus U^\vee$ ($U$ being a vector subspace of dimension $2$):
\[
\begin{pmatrix}
u & v \\
w &  - ^t u 
\end{pmatrix}
\]
with $u\in \mathrm{Hom}(U,U)$, and $v \in \mathrm{Hom}(U^\vee,U)$, $w\in \mathrm{Hom}(U,U^\vee)$ are symmetric. Let \[
S=\begin{pmatrix}
\mathrm{id}_V & 0 \\
0 &  - \mathrm{id}_{V^\vee} 
\end{pmatrix},
\]
then $\sigma \, : \, M \rightarrow SMS$ is an involution of $\Sp $ and for its adjoint group $PSp(4)$. Therefore $\Sp^\sigma \simeq U\otimes U^\vee =\mathfrak{gl}(U)$, and the wonderful compactification of the corresponding symmetric space $PSP(4)/GL(U)$ is of rank 2. We use the variety $Y$ introduced in Section \ref{section2} to describe its wonderful compactification.

\medskip

Let $W$ be the irreducible representation of $\mathfrak{so}(5)$ of dimension $5$. Recall that $\bigwedge^2 W \simeq \mathfrak{so}(5) \simeq \Sp \simeq S^2 V$, and $\bigwedge^2 V = W\oplus \co$. We denote $v \, : \, \mathbb{P}(V) \rightarrow \mathbb{P}(S^2 V)$ the Veronese embedding and $\mathcal{V}=v(\mathbb{P}(V))$, $\mathcal{V}\subset G(2,W)$ the grassmannian variety of plans in $W$. 

\begin{thm} \label{blowup}
The wonderful compatification of $\mathfrak{sp}(4)$ with maximal rank is isomorphic to $\tilde{G}$, the blowing up of the grassmannian $G(2,W)$ along the Veronese $\mathcal{V}$. 
\end{thm} 

Properties of blowing up show that $\tilde{G}$ is smooth and $\dim \tilde{G} = \dim  G(2,W) =6$. The main idea is to embed our two smooth varieties in the same projective spaces $\mathbb{P}(\Gamma_{2\rho}\oplus \Sp)$, that, we can find an automorphism $G$-invariant of $\mathbb{P}(\Gamma_{2\rho}\oplus \Sp)$ which sends one of them to the other one. Denote $\Gamma_{a,b}$  the irreducible representation of $\Sp$ with highest weight $a\omega_1+b\omega_2$ ($\omega_1$, $\omega_2$ are fundamental weights), for example $\Gamma_{2\rho}=\Gamma_{2,2}$.

\medskip

\textit{Remark}. Thanks to Lemma \ref{exactseq}, recall the exact sequence of $\Sp$-modules:
\[
0 \rightarrow \bigwedge^9 \Sp \stackrel{\delta}{\longrightarrow} \bigwedge^{6} \Sp / \Gamma_{2,2} \stackrel{\delta}{\longrightarrow} \bigwedge^3 \Sp \stackrel{\delta}{\longrightarrow} \co \rightarrow 0.
\]
For the proof of Theorem \ref{blowup}, we need the decomposition in irreducible representations of each term:
\begin{align*}
\bigwedge^6 \Sp = \Gamma_{2,2}\oplus \Gamma_{4,0}\oplus \Gamma_{0,3}\oplus \Gamma_{2,1}\oplus \Gamma_{0,2} \oplus \Gamma_{0,1} \oplus \Sp, \\
\bigwedge^3 \Sp = \Gamma_{4,0}\oplus \Gamma_{0,3}\oplus \Gamma_{2,1}\oplus \Gamma_{0,2}\oplus \Gamma_{0,1} \oplus \co.  
\end{align*}  

\begin{proof}[Proof of Theorem \ref{blowup}.]

\textbf{For the wonderful compactification.} Choose an element $U$ in the open orbit of the variety $Y$ and let $u$ be a representive of $U$ in $\bigwedge^6 \Sp$. Clearly $\delta(u)=0$, so thanks to the previous remark, $u$ is an element of $\Gamma_{2,2}\oplus\Sp$ so $Y\subset \mathbb{P}(\Gamma_{2,2}\oplus \Sp)$. If $c$ is the universal Casimir element of $\Sp$, $u$ and $c.u$ are independant (choose a suitable bases of $\Sp$ and compute the two elements), so $u$ does not lie in one irreducible representation ($c$ acts by a scalar on each irreducible representation).

\textbf{For the blowing up $\tilde G$}. As $GL(W)$-module, $\mathrm{H}^0(\mathcal{O}_{G(2,W)}(3))$ is isomorphic to the irreducible representation with partition $(3,3)$ (see Proposition 3.14 in \cite{We}). So using branching formulae in \cite{FH}, we have the decomposition on irreducible $\Sp $-modules $\mathrm{H}^0(\mathcal{O}_{G(2,W)}(3))=\Gamma_{6,0}\oplus \Gamma_{2,0}\oplus \Gamma_{2,2}$. Now, we use the exact sequence:
\[
0\rightarrow I_\mathcal{V}(3) \rightarrow \mathcal{O}_{G(2,W)}(3) \rightarrow \mathcal{O}_{\mathcal{V}}(3) \rightarrow 0,
\] 
where $I_\mathcal{V}$ is the sheaf of ideals which defines the Veronese in the grassmannian $G(2,W)$, and so,
\[
0\rightarrow \mathrm{H}^0 (I_\mathcal{V}(3)) \rightarrow \mathrm{H}^0(\mathcal{O}_{G(2,W)}(3)) \rightarrow  \mathrm{H}^0 (\mathcal{O}_{\mathcal{V}}(3)) \rightarrow \mathrm{H}^1 (I_{\mathcal{V}}(3)) \rightarrow  \cdots.
\]
Recall that $\mathrm{H}^0 (\mathcal{O}_\mathcal{V}(3))\simeq \mathrm{H}^0 (\mathcal{O}_{\mathbb{P}(V)}(6))=S^6 V= \Gamma_{6,0}$, so $\mathrm{H}^0 ( I_\mathcal{V}(3))=\Gamma_{2,2}\oplus \Gamma_{2,0}$. 
\\
The pullback of $I_\mathcal{V}(3)$ is a very ample sheaf of $\tilde{G}$. Indeed, denote by $Q_1$ the quotient sheaf of the grassmannian $G(2,W)$. The morphism $S^2 Q_1(1)\simeq S^2 Q_1^\vee(3) \rightarrow I_{\mathcal{V}}(3)$ is surjective, so $\mathcal{O}_{\mathrm{Proj}(I_\mathcal{V}(3))}(1)$ is the restriction of the very ample sheaf $\mathcal{O}_{\mathrm{Proj}(S^2 Q_1(1))}(1)$ to $\tilde{G}$.
\\
To conclude, $\tilde{G}$ is a subvariety of $\mathbb{P} \left( \mathrm{H}^0(I_\mathcal{V}(3)) \right) =\mathbb{P}(\Gamma_{2,2}\oplus \Sp ) $. 

\indent \textbf{The isomorphism}. The Veronese $\mathcal{V}$ and $G(2,W)$ are $G$-stable, so is $\tilde{G}$. The fact $V\simeq U\oplus U^\vee $ induces that $W \simeq \bigwedge^2 U \oplus \bigwedge^2 U^\vee \oplus \mathfrak{sl}(U)$. It follows that $\bigwedge^2 U \oplus \bigwedge^2 U^\vee$ is an element of the open orbit of $G(2,W)$, which means that its intersection with its orthogonal is reduced to zero, coming for a unique point of $\tilde{G}$. Denote by $[x]$ this point in $\mathbb{P}(\Gamma_{2,2}\oplus \Sp)$ which is invariant under the action of $GL(U)$.

Since $GL(2)\simeq G^\sigma$, $ \mathfrak{gl}(U) \subset \Sp $ is a point of the open orbit of $\bar{X}$, so $\mathfrak{gl}(U)^\bot$ is a point of the open orbit of $Y$, invariant under the action of $GL(U)$. Denote by $[y]$ this point in $\mathbb{P}(\Gamma_{2,2}\oplus \Sp)$.

The two points $x$ and $y$ have components only on the $GL(U)$ trivial factor of $\Gamma_{2,2}\oplus \Sp $: there is one trivial factor in $\Sp$; using Lemma \ref{exactseq}, and decomposing each space in irreducible $GL(U)$-module, we show that $\Gamma_{2,2}$ has another one. Now, $x$ and $y$ have non zero components on these two trivial factors (if it is not the case, $\tilde{G}$ or $Y$ can be embedded in some smaller $G$-stable projective space, that is to say $\mathbb{P}(\Gamma_{2,2})$ or $\mathbb{P}(\Sp)$. We also find two non zero complex numbers $\alpha$ and $\beta$ such that $\phi=\alpha \mathrm{id}_{\Gamma_{2,2}}+\beta \mathrm{id}_{\Sp}$ sends $x$ to $y$. The morphism $\phi$ is a $G$-equivariant automorphism of $\mathbb {P}(\Sp\oplus \Gamma_{2,2})$. Finally, $\tilde{G}$ and $Y$ are isomorphic.   
\end{proof}

\medskip

\textit{Acknowledgement.} Laurent Gruson is at the origin of this work; the author would like to thank him for his helpful remarks and advices.

\def\Dbar{\leavevmode\lower.6ex\hbox to 0pt{\hskip-.23ex \accent"16\hss}D}

\end{document}